\documentclass[letterpaper, 10pt, conference]{ieeeconf}
\IEEEoverridecommandlockouts
\usepackage[utf8]{inputenc}
\usepackage[T1]{fontenc}

\usepackage{amsthm}

\usepackage{cite}
\usepackage[hidelinks]{hyperref}
\hypersetup{
	pdftitle={Distributed Continuous-Time Optimization on the Special
		Orthogonal Group SO(3) over Tree Interaction Graphs},
	pdfauthor={Dongming Wang and Wei Ren}
}
\usepackage{graphicx}
\usepackage{xcolor}
\usepackage{balance}
\usepackage[theorems]{textool}
\newcommand{\accM}{1.14}
\newcommand{\accThr}{1.425}
\newcommand{\accThrTwo}{2.85}
\newcommand{\accMuF}{4.62}
\newcommand{\accRate}{0.462}
\newcommand{\accC}{0.46}
\newcommand{\accTol}{9.6\times10^{-3}}

\newcommand{\accSafety}{1.40}

\newcommand{\accTtolStar}{0.200}
\newcommand{\accBoundStar}{3.03}

\newcommand{\accSigminStar}{2.00}

\newcommand{\accRateFitStar}{0.495}

\newcommand{\accMinMargStar}{0.064}

\newcommand{\accTtolPath}{0.204}

\newcommand{\accSigminPath}{2.00}

\newcommand{\accRateFitPath}{0.496}

\newcommand{\accTtolTtree}{0.200}

\newcommand{\accSigminTtree}{2.00}

\newcommand{\accRateFitTtree}{0.497}

\newcommand{\accEps}{2.4\times10^{-3}}
\newcommand{\accSigTwoTimeStar}{0.15}

\newcommand{\accWslopeStratumStar}{2.67}

\newcommand{\accSigNearTwoTimePath}{0.15}

\newcommand{\accSigNearExcessPath}{0.15}

\newcommand{\accSigTwoTimeTtree}{0.11}

\newcommand{\accSlideWideOneStar}{1.4\times10^{-3}}
\newcommand{\accSlideWideTwoStar}{2.2\times10^{-3}}

\newcommand{\accSlideWideOnePath}{7.6\times10^{-4}}
\newcommand{\accSlideWideTwoPath}{1.4\times10^{-3}}

\newcommand{\accSlideWideOneTtree}{9.8\times10^{-4}}
\newcommand{\accSlideWideTwoTtree}{1.7\times10^{-3}}

\newcommand{\accSlideWideOneAll}{1.4\times10^{-3}}

\newcommand{\accSlideNarrowOneAll}{2.2\times10^{-1}}

\newcommand{\accHalfDiffMax}{3.5}

\newcommand{\accPropWsix}{1.7\times10^{-1}}

\newcommand{\accSweepAlphaMin}{1.0}
\newcommand{\accSweepAlphaMax}{4.0}

\newcommand{\accSweepTtolLoStar}{0.262}
\newcommand{\accSweepTtolHiStar}{0.103}

\newcommand{\accSweepTreeSpread}{3}

\newcommand{\accRadRatioMaxAll}{0.066}

\newcommand{\accRdotsB}{-5.81, -1.11, -2.24, -1.91, -1.32}

\newcommand{\accTtolB}{0.347}
\newcommand{\accMinMargB}{5.6\times10^{-5}}
\newcommand{\accSepZero}{5.3}
\newcommand{\accLeafGrad}{-0.16}
\newcommand{\accTtolC}{0.377}
\newcommand{\accMinMargC}{8.0\times10^{-6}}
\newcommand{\accHb}{5\times10^{-5}}

\newcommand{\accRndRatioLoFive}{0.032}
\newcommand{\accRndRatioHiFive}{0.088}

\newcommand{\accRndRatioLoEight}{0.016}
\newcommand{\accRndRatioHiEight}{0.037}

\newcommand{\accRndRatioLoTwelve}{0.007}
\newcommand{\accRndRatioHiTwelve}{0.013}

\DeclareMathOperator{\SO}{SO}
\DeclareMathOperator{\so}{\mathfrak{so}}
\DeclareMathOperator{\Ad}{Ad}
\newcommand{\sgn}{\mathrm{sgn}}
\newcommand{\Sgn}{\mathrm{SGN}}

\title{\LARGE\bf Distributed Continuous-Time Optimization on the
	Special Orthogonal Group $\SO(3)$ over Tree Interaction Graphs}

\author{Dongming Wang$^{1}$, Wei Ren$^{1}$%
	\thanks{$^{1}$D.~Wang and W.~Ren are with the Department of
		Electrical and Computer Engineering, University of California,
		Riverside, CA 92521, USA.
		{\tt\small \{wdong025,wei.ren\}@ucr.edu}.}%
}

\begin{document}
	\maketitle

	\begin{abstract}
		We study continuous-time distributed optimization on the
		three-dimensional rotation group over undirected tree graphs,
		where agents with heterogeneous local costs seek a common
		attitude minimizing the geodesically strongly convex sum of
		their costs on a geodesically convex operating ball.
		We propose a nonsmooth protocol with fixed gains that combines local Riemannian gradient descent with signum-based consensus feedback computed from the Lie-group logarithms of relative rotations between neighboring agents.
		Since geodesic strong convexity alone does not guarantee
		invariance, as demonstrated by a certified counterexample,
		we impose an inward-pointing boundary condition and prove
		that every Filippov solution with all agents initialized in
		the ball remains there.
		We construct a nonsmooth intrinsic disagreement Lyapunov
		function and derive a uniform dissipation bound through
		a cluster-contraction argument exploiting the tree structure.
		With bounded local gradients, this bound gives an explicit
		gain-ratio condition under which we establish exact finite-time
		consensus for every such solution and obtain a settling-time
		estimate requiring no structural graph constant beyond the
		number of agents.
		After consensus, we identify the sliding dynamics as a scaled
		Riemannian gradient flow of the summed objective.
		When the minimizer lies in the interior of the ball, the
		dynamics converge to it exponentially in geodesic distance.
		Simulations illustrate our theoretical results.
	\end{abstract}

	\section{Introduction}

	Networks of rigid bodies, including spacecraft formations, camera
	networks, and robotic swarms, often require their agents to coordinate
	on a common attitude~\cite{sarlette2009consensus,
		tron2013riemannian}. In many applications,
	however, the agents may not share the same local objective: individual
	platforms may encode different pointing preferences, sensor
	calibration requirements, or tracking targets. This leads naturally
	to a distributed optimization problem in which the agents must agree
	on one attitude while minimizing the sum of heterogeneous local
	costs. The decision variable naturally lies on the rotation group $\SO(3)$, a
	compact nonlinear manifold for which Euclidean
	consensus--optimization methods cannot be applied
	directly. Existing results on optimization on manifolds are primarily centralized~\cite{absil2008optimization,boumal2023intromanifolds}.

	Consensus and synchronization on Riemannian manifolds have been
	studied under various geometric and communication assumptions.
	Finite-time consensus on strongly convex Riemannian balls is
	studied in~\cite{chen2014finite}. Finite-time protocols on
	$\SO(3)$ include discontinuous~\cite{wei2017finite}, continuous
	non-Lipschitz~\cite{maadani2020finite}, and
	adaptive~\cite{tang2026adaptive} designs. 	For heterogeneous rigid bodies over undirected trees,
	\cite{boughellaba2026global} achieves almost-global and global attitude synchronization via continuous and hybrid designs. These synchronization guarantees do not establish convergence
	to a minimizer of a prescribed sum of heterogeneous local costs.
	Introducing local cost-gradient terms changes the coupled
	dynamics, so convergence to an optimizer requires additional and more complicated
	analysis beyond the synchronization results.

	Distributed minimization of sums of local costs on Riemannian
	manifolds has been studied through intrinsic iterative
	methods~\cite{shah2017distributed}. Refs.~\cite{deng2025decentralized} and \cite{deng2025stochastic} develop
	projected methods for smooth nonconvex optimization on compact
	embedded submanifolds, including a constant-step
	gradient-tracking method and a stochastic recursive-momentum
	method, with first-order stationarity
	guarantees.
	Ref.~\cite{nguyen2026intrinsic} establishes sublinear
	optimality-gap bounds for intrinsic stochastic iterations with
	diminishing gradient steps for geodesically convex local costs.
	Because $\SO(3)$ is a compact embedded Riemannian manifold,
	these discrete-time methods can be implemented on $\SO(3)$
	when their respective smoothness, communication, stepsize,
	stochastic-oracle, and domain assumptions hold. However, they do not directly provide continuous-time convergence
	guarantees; in particular, translating the projected updates
	into continuous-time dynamics requires a separate formulation
	and analysis. Moreover, projection onto $\SO(3)$ enforces
	manifold membership but does not by itself ensure forward
	invariance of a prescribed geodesic operating ball, whereas
	the intrinsic stochastic analysis assumes, rather than proves,
	confinement of both the iterates and intermediate iterates
	to the convex domain.
	An additional invariance argument is therefore needed, since
	geodesic strong convexity alone does not guarantee confinement
	to the operating ball~\cite{wang2026counterexample}.

	On $\SO(3)$, \cite{li2025semidecentralized} studies a
	continuous-time consensus--gradient algorithm on a star with a
	designated coordinator and a diminishing gradient gain,
	directly assuming that trajectories remain in the prescribed convex
	region. Ref.~\cite{kraisler2023lie} proposes an
	intrinsic continuous-time gradient-tracking method for
	Riemannian center-of-mass optimization on Lie groups.
	Their analysis identifies the optimal limit conditional on
	convergence within the prescribed ball, but the convergence theorem
	is only established for Euclidean space.
	We complement these results by establishing
	confinement and convergence on \(\SO(3)\) for a fixed-gain distributed continuous-time
	angular-velocity law on arbitrary undirected trees under explicit
	boundary and gain conditions. Strong convexity is required only
	of the summed objective, whose minimizer is assumed to lie in
	the interior of the operating ball.

	In Euclidean space, our previous work~\cite{lin2017distributed}
	combines signum consensus feedback with local gradients to achieve
	finite-time consensus and asymptotic minimization of the sum
	of local costs.
	Inspired by this idea, we propose an intrinsic fixed-gain
	protocol that combines signum feedback based on the logarithms
	of relative rotations between neighbors with local Riemannian
	gradient descent.
	However, the curved geometry of \(\SO(3)\) poses a distinctive analytical challenge that prevents direct application of the Euclidean convergence analysis and calls for novel ones. We address this challenge by constructing a nonsmooth intrinsic Lyapunov function and proving its strict decay before consensus via a joint Filippov analysis accounting for heterogeneous local gradients and simultaneous zero-error edges.

	\emph{The main contributions are as follows.}

	(i)~We establish strong forward invariance of the prescribed
	operating ball under an inward-pointing boundary condition,
	ensuring that every Filippov solution remains in the domain
	where the geometric and convexity assumptions apply.
	The trajectory-level argument remains valid even when
	pointwise containment of the Filippov set in the boundary
	tangent cone fails. We verify the cost assumptions
	constructively for weighted squared-distance costs and
	explain why strong convexity alone does not replace the
	boundary condition for an off-center operating ball;
	see Theorem~\ref{thm:main}(i) and
	Remarks~\ref{rem:class}, \ref{rem:necessity},
	and~\ref{rem:viab}.
	
	(ii)~Within this invariant domain, we construct a nonsmooth
	intrinsic disagreement Lyapunov function based on neighboring
	geodesic distances and analyze its evolution under the joint
	Filippov dynamics. Aggregating agents connected by zero-error
	edges cancels the internal set-valued terms, allowing
	simultaneous edge coincidences to be treated rigorously.
	The contracted tree supplies a uniform lower bound
	on the sum of cluster-aggregate signum norms, yielding a
	dissipation estimate that holds almost everywhere along
	every solution, including during partial agreement;
	see Lemmas~\ref{lem:clusters} and~\ref{lem:tree_bound}
	and Step~2 of the proof of Theorem~\ref{thm:main}.
	
	(iii)~Combining invariance with this Lyapunov analysis,
	we establish a unified convergence theorem for the intrinsic
	signum--gradient protocol with fixed gains on undirected
	trees. Under an explicit gain-ratio condition, every
	Filippov solution initialized in the operating ball reaches
	and maintains exact finite-time consensus, with a
	settling-time estimate requiring no structural graph
	constant beyond the number of agents.
	After consensus, we identify the common motion as a scaled
	Riemannian gradient flow of the summed objective and establish
	exponential convergence to its unique minimizer through an
	explicit distance estimate;
	see Theorem~\ref{thm:main}(ii)--(iii).

	Section~\ref{sec:sim} numerically illustrates the cluster bound,
	settling-time law, sliding dynamics, and invariance result.

	\section{Preliminaries}\label{sec:prelim}

	\subsection{Geometry of $\SO(3)$}

	The three-dimensional special orthogonal group $\SO(3):=\{R\in\mbR^{3\times3}:R^\top R=I_3,\ \det R=1\}$ has Lie
	algebra $\so(3):=\{\Omega\in\mbR^{3\times3}:\Omega^\top=-\Omega\}$,
	identified with $\mbR^3$ through the hat map
	$\xi^\wedge \eta=\xi\times \eta$ and its inverse
	$(\cdot)^\vee$, for \(\xi\), \(\eta \in \mbR^3\)~\cite{hall2015lie,sola2018micro}.
	Under the bi-invariant metric
	$\langle R\xi^\wedge,R\eta^\wedge\rangle_R=\xi^\top\eta$,
	geodesics through $R$ are $t\mapsto R\exp(t\xi^\wedge)$, \(t\in\mbR\), and for \(Q\in \SO(3)\), the
	induced distance satisfies
	$d(R,Q)=\|(\log(R^\top Q))^\vee\|$ whenever $d(R,Q)<\pi$.
	With this normalization, $\SO(3)$ has constant sectional curvature
	$1/4$~\cite{tron2013riemannian}.
	For a center $R_c\in\SO(3)$, the closed geodesic ball
	$\overline{\ccalB}_\rho(R_c):=\{R:d(R,R_c)\leq\rho\}$ is strongly
	geodesically convex when $\rho<\pi/2$. Thus, any
	$R_i,R_j\in\overline{\ccalB}_\rho(R_c)$ satisfy
	$d(R_i,R_j)\leq2\rho<\pi$, so the relative error
	$e_{ij}:=(\log(R_i^\top R_j))^\vee$ is smooth and single-valued,
	with $\|e_{ij}\|=d(R_i,R_j)$ and $e_{ji}=-e_{ij}$. For
	$e_{ij}\neq0$, let $\hat e_{ij}:=e_{ij}/\|e_{ij}\|$.

	Let $R_i,R_j$ evolve under $\dot R_k=R_k\omega_k^\wedge$ with
	$\|e_{ij}\|<\pi$. Then, wherever $e_{ij}\neq 0$,
	\feq{
		\frac{d}{dt}\tfrac12\|e_{ij}\|^2
		= e_{ij}^\top(\omega_j-\omega_i),
		\quad
		\frac{d}{dt}\|e_{ij}\|
		= \hat e_{ij}^\top(\omega_j-\omega_i).
		\label{eq:deriv}
	}
	Intrinsically, \eqref{eq:deriv} is the first-variation formula for
	$\tfrac12 d^2(\cdot,\cdot)$ specialized to a bi-invariant Lie
	group~\cite{docarmo1992riemannian}; we record the short
	body-frame verification.
	With $\phi:=e_{ij}$ and the right Jacobian
	$J_r^{-1}(\phi)=I+\tfrac12\phi^\wedge+\beta(\|\phi\|)
	(\phi^\wedge)^2$, where
	$\beta(\theta):=\theta^{-2}\bigl(1-\tfrac{\theta}{2}
	\cot\tfrac{\theta}{2}\bigr)$ (only its scalar nature is used
	below), the relative-rotation kinematics give
	$\dot\phi=J_r^{-1}(\phi)
	(\omega_j-\Ad_{R_j^\top R_i}\omega_i)$, where
	$\Ad_Q\omega:=Q\omega$ denotes the adjoint action of
	$Q\in\SO(3)$ in vector coordinates~\cite{sola2018micro,
		chirikjian2011stochastic}.
	Then $\tfrac{d}{dt}\tfrac12\|\phi\|^2=\phi^\top\dot\phi$.
	The identities $\phi^\top\phi^\wedge=0$ and
	$\phi^\top(\phi^\wedge)^2=0$ eliminate the Jacobian corrections,
	while $\phi^\top\Ad_{R_j^\top R_i}\omega_i=\phi^\top\omega_i$
	because $R_j^\top R_i=\exp(-\phi^\wedge)$ fixes $\phi$.
	This yields the first identity in~\eqref{eq:deriv}; division by
	$\|\phi\|$ gives the second.

	\begin{lemma}[Quantitative halfspace property]\label{lem:half}
		Let $r\in(0,\pi/2)$, $R\in\partial\overline{\ccalB}_r(R_c)$,
		where \(\partial\) denotes the boundary,
		$Q\in\overline{\ccalB}_r(R_c)$ with $L:=d(R,Q)>0$, and let
		$\hat e_{c},\hat e_{Q}$ be the unit body-coordinate directions at $R$
		toward $R_c$ and $Q$.
		Then
		$\hat e_c^\top\hat e_Q\geq
		\cot\tfrac{r}{2}\tan\tfrac{L}{4}>0$.
	\end{lemma}
	\begin{proof}
		$\SO(3)$ with the bi-invariant metric has \emph{constant}
		sectional curvature $1/4$, so within the injectivity region the
		exact spherical law of cosines of the model sphere (radius $2$)
		applies~\cite{docarmo1992riemannian}: with $c:=d(R_c,Q)$,
		$\cos\tfrac{c}{2}=\cos\tfrac{r}{2}\cos\tfrac{L}{2}
		+\sin\tfrac{r}{2}\sin\tfrac{L}{2}\,
		(\hat e_c^\top\hat e_Q)$.
		Solving and using $c\leq r$,
		\(
		\hat e_c^\top\hat e_Q
		\geq\frac{\cos\tfrac{r}{2}
			\bigl(1-\cos\tfrac{L}{2}\bigr)}
		{\sin\tfrac{r}{2}\sin\tfrac{L}{2}}
		=\cot\tfrac{r}{2}\tan\tfrac{L}{4},
		\)
		by $(1-\cos x)/\sin x=\tan(x/2)$; positivity follows from
		$0<L\leq2r<\pi$.
		Equality holds precisely when $c=r$.
	\end{proof}

	\subsection{Strong convexity}

	For a twice continuously differentiable (\(C^2\)) function \(f\)
	on an open subset of $\SO(3)$, let $\nabla_Rf\in\mbR^3$ denote its
	body-coordinate Riemannian gradient, so that
	$\operatorname{grad}f(R)=R(\nabla_Rf)^\wedge$, and let $H_R$ be
	the body-coordinate matrix of its Riemannian
	Hessian~\cite{boumal2023intromanifolds}. Then $f$ is geodesically
	$\mu$-strongly convex on a geodesically convex set $\ccalC$ if
	$H_R\succeq\mu I_3$ for every $R\in\ccalC$.
	If additionally the minimizer $R^*=\argmin_\ccalC f$ lies in the
	\emph{interior} of $\ccalC$, then $\nabla f(R^*)=0$, and combining
	the strong-convexity inequality at $R$ and at $R^*$ yields
	\feq{
		\inner{\nabla_R f}{(\log(R^\top R^*))^\vee}
		\leq -\mu\, d^2(R,R^*),
		\quad\forall R\in\ccalC.
		\label{eq:sc}
	}

	\subsection{Filippov solutions}

	Let $\ccalD\subset\SO(3)^n$ be open, and let
	$h:\ccalD\to\mbR^{3n}$ be the global left-trivialized
	body-coordinate representation of a locally bounded measurable
	vector field. For $x\in\ccalD$, let $\ccalB_\delta(x)$ denote the
	radius-$\delta$ geodesic neighborhood in $\ccalD$, let $\nu$ denote
	Riemannian volume, and let $\overline{\mathrm{co}}$ denote closed
	convex hull. The Filippov map is
	$\ccalK[h](x):=\bigcap_{\delta>0}\bigcap_{\nu(\ccalS)=0}
	\overline{\mathrm{co}}\,h(\ccalB_\delta(x)\setminus\ccalS)$, where
	$\ccalS$ ranges over measurable Riemannian-null sets
	~\cite{filippov1988differential,cortes2008discontinuous}.
	All differential inclusions below are expressed in this
	left trivialization; continuous chart and frame factors may be
	evaluated at $x$ in the Filippov limit: a smooth frame change
	multiplies $h$ by a continuous invertible matrix function, with
	which the Filippov regularization
	commutes~\cite{cortes2008discontinuous}, so $\ccalK[h]$ is
	frame-independent and values of $h$ at nearby points may be
	compared in their own body frames. The map $\ccalK[h]$ is
	upper semicontinuous with nonempty compact convex values.
	For $z\in\mbR^3$, define
	$\overline{\mbB}^3:=\{u\in\mbR^3:\|u\|\leq1\}$ and
	$\Sgn(z):=\{z/\|z\|\}$ for $z\neq0$, with
	$\Sgn(0):=\overline{\mbB}^3$. The control law uses the single-valued
	selection $\sgn(z):=z/\|z\|$ for $z\neq0$ and $\sgn(0):=0$.

	\section{Problem Statement and Assumptions}\label{sec:problem}

	Let $\ccalV:=\{1,\ldots,n\}$ index agents with attitudes
	$R_i\in\SO(3)$, body angular velocities $\omega_i\in\mbR^3$, and
	kinematics $\dot R_i=R_i\omega_i^\wedge$. They communicate over a
	fixed undirected graph $\ccalG=(\ccalV,\ccalE)$, with neighbor set
	$\ccalN_i:=\{j:(i,j)\in\ccalE\}$. Fix $R_c\in\SO(3)$ and
	$\rho\in(0,\pi/2)$, and define the operating region
	$\ccalC:=\overline{\ccalB}_\rho(R_c)$. Each agent holds a $C^2$
	cost $f_i:\ccalO\to\mbR$, where $\ccalO\supset\ccalC$ is open,
	with body-coordinate gradient $\nabla_Rf_i$ as in
	Section~\ref{sec:prelim}.

	The objective is to design the velocities $\omega_i$ so that the
	agents reach exact consensus and their common attitude converges to
	a solution of
	\feq{
		\min_{R\in\ccalC}\ F(R):=\sum_{i=1}^n f_i(R).
		\label{eq:obj}
	}
	At run time, agent $i$ evaluates $\nabla_{R_i}f_i$ and measures
	$R_i^\top R_j$ for $j\in\ccalN_i$; the common gains are selected
	offline using the agent count $n$ and the uniform gradient bound
	$M$ defined below.

	\begin{assumption}[Tree topology]\label{ass:tree}
		$\ccalG$ is a tree: connected with $|\ccalE|=n-1$.
	\end{assumption}

	\begin{assumption}[Convexity, gradients, and initialization]
		\label{ass:conv}
		On $\ccalC$: (a)~$F$ is geodesically
		$\mu_F$-strongly convex, with $\mu_F>0$;
		(b)~$R^*:=\argmin_{R\in\ccalC}F(R)$ lies in the interior of
		$\ccalC$; (c)~$R_i(0)\in\ccalC$ for every $i$. Fix any $M>0$
		satisfying $M\geq\max_{i\in\ccalV}\max_{R\in\ccalC}
		\|\nabla_Rf_i(R)\|$; such an $M$ exists since each $f_i$ is
		$C^2$ on the open set $\ccalO\supset\ccalC$ and $\ccalC$ is
		compact.
	\end{assumption}

	\begin{assumption}[Inward-pointing descent at the
		boundary]\label{ass:rad}
		$\hat e_{i,c}^\top\nabla_{R_i}f_i\leq 0$ for every
		$R_i\in\partial\ccalC$ and every $i$, where
		$\hat e_{i,c}$ is the unit direction from $R_i$ toward $R_c$.
	\end{assumption}

	Throughout, $\hat e_{i,c}$ points inward, so the inequality
	states that the descent velocity $-\nabla_{R_i}f_i$ is weakly
	inward at the boundary.

	\begin{remark}[Verifiable cost class]\label{rem:class}
		For weighted squared distances
		$f_i=\tfrac{k_i}{2}d^2(\cdot,R_i^*)$ with $k_i>0$ and targets
		$R_i^*\in\overline{\ccalB}_{r_0}(R_c)$, $r_0<\rho$, all
			cost-related conditions in Assumptions~\ref{ass:conv}
			and~\ref{ass:rad} are verified constructively; indeed,
			$\rho+r_0<\pi$, so the costs are $C^2$ on a common open
			neighborhood of $\ccalC$.
		The distance identity~\eqref{eq:deriv} with $R_j\equiv R_i^*$ static gives
		$\nabla_{R}f_i=-k_i\,e_{i,R_i^*}$; Assumption~\ref{ass:rad} then
		follows from Lemma~\ref{lem:half}, and
		$\|\nabla f_i\|=k_i\,d(R,R_i^*)\leq k_i(\rho+r_0)=:M_i$ gives
		$M=\max_i M_i$.
		The classical Jacobi-field computation on the constant-curvature
		space $\kappa=1/4$~\cite{docarmo1992riemannian,
			tron2013riemannian} gives the Hessian spectrum of
		$\tfrac12 d^2(\cdot,p)$ as $\{1,\ \tfrac{r}{2}\cot\tfrac{r}{2}\}$
		at distance $r$, whence one may take
			$\mu_F=(\sum_i k_i)\tfrac{\rho+r_0}{2}
			\cot\tfrac{\rho+r_0}{2}>0$.
		Finally, interiority~(b) holds automatically, by a nested-ball
		argument: for any $R\in\ccalC$ with $r:=d(R,R_c)>r_0$, every target lies
		strictly inside $\overline{\ccalB}_r(R_c)$ (at distance
		$\leq r_0<r$ from $R_c$), so Lemma~\ref{lem:half} applied to
		the ball of radius $r$ gives
		$\hat e_c^\top\hat e_{i,R_i^*}\geq
		\cot\tfrac{r}{2}\tan\tfrac{L_i}{4}>0$ for every target with
		$L_i=d(R,R_i^*)>0$; hence
		$\hat e_c^\top\nabla_R F=-\sum_ik_iL_i\,
		\hat e_c^\top\hat e_{i,R_i^*}<0$, $F$ strictly decreases
		toward $R_c$ at every point outside
		$\overline{\ccalB}_{r_0}(R_c)$, and its minimizer over
		$\ccalC$ satisfies $d(R^*,R_c)\leq r_0<\rho$.
	\end{remark}

	\begin{remark}[Why a boundary condition is needed]
		\label{rem:necessity}
		Assumption~\ref{ass:rad} is the natural hypothesis for an
		operating ball not centered at the minimizer, and strong
		convexity does not imply it. By~\eqref{eq:sc} the descent
		direction points toward $R^*$, so geodesic balls centered at
		$R^*$ are invariant under the gradient flow; for an off-center
		ball, anisotropy alone breaks invariance. Already in $\mbR^2$,
		for $f(x)=\tfrac12(x-x^*)^\top\operatorname{diag}(1,10)(x-x^*)$
		on the ball of radius $\rho$ about $c$ with
		$x^*=c+(0,0.9\rho)$, the descent direction at the boundary
		point $c+\rho(\cos\theta,\sin\theta)$ has outward normal
		component $\rho\,(9\sin\theta-9\sin^2\theta-1)>0$ for
		$\sin\theta\in(0.13,0.87)$. The obstruction concerns one
		gradient flow at the boundary and is independent of the
		network; what is specific to $\SO(3)$ is only that an operating
		ball is needed at all, and~\cite{wang2026counterexample}
		certifies that the same mechanism persists under curvature.
	\end{remark}

	\section{Protocol and Main Result}\label{sec:protocol}

	The distributed control law is
	\feq{
		\omega_i
		= \alpha\!\!\sum_{j\in\ccalN_i}\!\sgn(e_{ij})
		- \gamma\,\nabla_{R_i}f_i,
		\quad \alpha,\gamma>0,
		\label{eq:law}
	}
	with $\sgn(0)=0$; write
	$s_i:=\sum_{j\in\ccalN_i}\sgn(e_{ij})$.
	Let $x:=(R_1,\ldots,R_n)$ and define the joint left-trivialized
	body-coordinate field \(h(x):=
	\bigl(\alpha s_i-\gamma\nabla_{R_i}f_i\bigr)_{i\in\ccalV}\).
	The closed loop is interpreted through its Filippov regularization
	$\ccalK[h]$ on the domain \(\ccalD_\ccalO:=
	\bigl\{x=(R_i)_{i\in\ccalV}\in\ccalO^n:
	d(R_i,R_j)<\pi,~\forall (i,j)\in\ccalE\bigr\}\).
	Here $\ccalC^n\subset\ccalD_\ccalO$ because pairwise distances in
	$\ccalC$ are at most $2\rho<\pi$.

	\begin{definition}[Filippov solution]\label{def:viable}
		A \emph{Filippov solution} on $[0,T]$ is an absolutely continuous
		map $x:[0,T]\to\ccalD_\ccalO$ whose body-frame velocities satisfy
		$\bigl((R_i^\top\dot R_i)^\vee\bigr)_i\in\ccalK[h](x)$ almost
		everywhere. Throughout, ``solution'' refers to a Filippov solution
		of the closed loop.
	\end{definition}

	The field $h$ is locally bounded on $\ccalD_\ccalO$.
	Its Filippov regularization has nonempty compact convex
	values and is upper semicontinuous, so a local solution
	exists from every initial condition in
	$\ccalD_\ccalO$~\cite{filippov1988differential,
		cortes2008discontinuous}.
	Since $\ccalK[h](x)$ lies in the closed convex hull of nearby
	values of $h$, solutions have bounded body velocities on compact
	subsets of $\ccalD_\ccalO$.

	Define the aggregate disagreement
	\[
	W(x):=\sum_{(i,j)\in\ccalE}\|e_{ij}\|,
	\]
	and abbreviate $W(0):=W(x(0))$.
	The following theorem states the main result.

	\begin{theorem}[Invariance and distributed optimization on trees]
		\label{thm:main}
		Suppose Assumptions~\ref{ass:tree}--\ref{ass:rad} hold,
		and let $\alpha,\gamma>0$. For the protocol~\eqref{eq:law},
		the following statements hold.
		\begin{enumerate}
			\item[(i)] \emph{Strong invariance and global existence.}
			From every initial condition $x(0)\in\ccalC^n$, a Filippov
			solution exists. Every maximal solution from this initial
			condition is defined on $[0,\infty)$ and satisfies
			$x(t)\in\ccalC^n$ for all $t\geq0$.
			\item[(ii)] \emph{Exact finite-time consensus.}
			If, additionally,
			\begin{equation}
				\frac{\alpha}{\gamma}>\frac{nM}{2},
				\label{eq:gain}
			\end{equation}
			then every such solution reaches and maintains exact
			consensus:
			\[
			R_i(t)=R_j(t),\qquad
			\forall i,j\in\ccalV,\quad t\geq T^*,
			\]
			for a settling time satisfying
			\begin{equation}
				T^*\leq
				\frac{nW(0)}{2(2\alpha-\gamma nM)}.
				\label{eq:settle}
			\end{equation}
			\item[(iii)] \emph{Aggregate gradient flow and exponential
				convergence.}
			Under~\eqref{eq:gain}, let $\bar R(t):=R_i(t)$ for
			$t\geq T^*$. Then
			\begin{equation}
				\dot{\bar R}
				=
				\bar R\Bigl(-\frac{\gamma}{n}
				\nabla_{\bar R}F\Bigr)^{\!\wedge}
				\quad\text{for a.e. }t\geq T^*,
				\label{eq:sliding}
			\end{equation}
			and
			\begin{equation}
				d(\bar R(t),R^*)
				\leq
				d(\bar R(T^*),R^*)\,
				e^{-\frac{\gamma\mu_F}{n}(t-T^*)},
				\quad t\geq T^*.
				\label{eq:rate}
			\end{equation}
			Here $T^*$ and $\bar R(T^*)$ may depend on the particular
			Filippov solution; for each solution, the common trajectory
			on $[T^*,\infty)$ is uniquely determined by $\bar R(T^*)$.
		\end{enumerate}
	\end{theorem}

	\begin{corollary}[Accuracy time]\label{cor:eps}
		Under the hypotheses of Theorem~\ref{thm:main}(ii), let
		$\rho^*:=d(R^*,R_c)<\rho$. Then, for every
		$\epsilon\in(0,\rho+\rho^*]$ and every $i\in\ccalV$,
		$d(R_i(t),R^*)\leq\epsilon$ whenever
		\feq{
			t\geq T^*+\frac{n}{\gamma\mu_F}
			\ln\frac{\rho+\rho^*}{\epsilon}.
			\label{eq:epstime}
		}
	\end{corollary}
	\begin{proof}
		By Theorem~\ref{thm:main}(i), $\bar R(T^*)\in\ccalC$, so the
		triangle inequality gives
		$d(\bar R(T^*),R^*)\leq\rho+\rho^*$, and~\eqref{eq:rate}
		yields the claim.
	\end{proof}

	\begin{remark}[On the gain condition]\label{rem:gain}
		Condition~\eqref{eq:gain} requires the global constants $n$ and
		$M$ for the design of $\alpha/\gamma$; $W(0)$ enters only the
		estimate~\eqref{eq:settle}. The geometric constant
		in~\eqref{eq:two} is sharp: two clusters joined by one active
		edge give $\sum_C\|S_C\|=2$. The settling estimate also uses
		the weighted Cauchy--Schwarz step and the uniform local-gradient
		bound $M$. The resulting gain condition is sufficient; its
		optimality under the stated assumptions is not established
		here.
	\end{remark}

	\section{Proof of the Main Result}\label{sec:analysis}

	\subsection{Filippov structure and cluster dynamics}

	\begin{definition}[Clusters]\label{def:clusters}
		At time $t$, the \emph{zero-error subgraph} is
		$\ccalG_0(t)=(\ccalV,\{(i,j)\in\ccalE: e_{ij}(t)=0\})$;
		its connected components form the \emph{cluster partition}
		$\ccalP_0(x(t))$.
		Adjacent agents with $e_{ij}=0$ coincide, so each cluster shares
		one attitude, and every edge between distinct clusters carries a
		nonzero error.
	\end{definition}

	\begin{lemma}[Joint Filippov structure]\label{lem:filippov}
		Let $x\in\ccalD_\ccalO$, with all neighborhoods below taken
		inside $\ccalD_\ccalO$.
		(i)~\emph{Antisymmetric containment:}
		\feq{
			\ccalK[h](x)\subseteq A(x):=\Bigl\{
			\Bigl(\alpha\!\sum_{j\in\ccalN_i}\!u_{ij}
			-\gamma\nabla_{R_i}f_i\Bigr)_{\!i}\Bigr\},
			\label{eq:containment}
		}
		where the family $u$ ranges over all assignments with
		$u_{ij}\in\overline{\mbB}^3$, $u_{ji}=-u_{ij}$ for every edge,
		and $u_{ij}=\sgn(e_{ij})$ whenever $e_{ij}\neq 0$.
		(ii)~\emph{Nominal membership:} the nominal value
		$h(x)=\bigl(\alpha s_i-\gamma\nabla_{R_i}f_i\bigr)_i$ (with
		$\sgn(0)=0$) belongs to $\ccalK[h](x)$.
		(iii)~\emph{Isolated-pair membership:} if the zero-error cluster
		containing a zero edge $(i_0,j_0)$ is exactly the pair
		$\{i_0,j_0\}$, then for any unit $u\in\mbR^3$ the value obtained
		from $h(x)$ by setting $u_{i_0j_0}=u=-u_{j_0i_0}$ on that edge,
		with $u_{ij}=0$ on every other zero edge, also belongs to
		$\ccalK[h](x)$.
	\end{lemma}
	\begin{proof}
		The set $A(x)$ is the image of a compact convex product of balls
		and points under an affine map, hence compact and convex.
		Shrinking $\delta$ so that every edge with $e_{ij}(x)\neq0$
		stays nonzero on $\ccalB_\delta(x)$, each value $h(y)$,
		$y\in\ccalB_\delta(x)$, has the form~\eqref{eq:containment}
		with the antisymmetric choice $u_{ij}=\sgn(e_{ij}(y))$ and
		gradients evaluated at $y$. Replacing $\sgn(e_{ij}(y))$ by
		$\sgn(e_{ij}(x))$ on the edges with $e_{ij}(x)\neq0$ and
		$\nabla f_i(y)$ by $\nabla f_i(x)$ gives an element of $A(x)$
		(the zero-edge slots are free), with an error that tends to
		$0$ uniformly on $\ccalB_\delta(x)$ as $\delta\to0$, by
		continuity of the nonzero-edge directions and of the
		gradients. Hence the essential values of $h$ near $x$ lie in a
		shrinking neighborhood of the closed convex set $A(x)$; taking
		closed convex hulls and intersecting over $\delta$ gives
		$\ccalK[h](x)\subseteq A(x)$, proving~(i).

		For~(ii), fix $\delta_1,\ldots,\delta_n\in\mbR^3$ with
		$\delta_j\neq\delta_i$ for every zero edge $(i,j)$ (generic),
		and for $\epsilon>0$ consider the \emph{full-dimensional}
		perturbation families
		$x^{\pm}(\epsilon,w):
		R_i\mapsto R_i\exp\bigl((\pm\epsilon\delta_i+w_i)^\wedge\bigr)$
		over $\norm{w}:=\max_i\norm{w_i}\leq\epsilon^2$.
		A fixed curve would be a null set, possibly inside the excluded
		set; the family is not: the image of the open $w$-ball has
		positive Riemannian measure, so for \emph{every} null set
		$\ccalS$ and every $\epsilon$ there are parameters $w^{\pm}$
		with $x^{\pm}(\epsilon,w^{\pm})\notin\ccalS$.
		Uniformly over $\norm{w}\leq\epsilon^2$ and over the finitely
		many zero edges, the BCH expansion
		gives $e_{ij}^{\pm}=\pm\epsilon(\delta_j-\delta_i)
		+O(\epsilon^2)$ on zero edges, while every edge that is nonzero
		at $x$ remains nonzero for
		sufficiently small $\epsilon$, and its signum direction converges
		to its value at $x$; hence
		$h\bigl(x^{\pm}(\epsilon,w^{\pm})\bigr)\to V^{\pm}$ as
		$\epsilon\to0$, where $V^{\pm}$ are the values of the
		right-hand side of~\eqref{eq:containment} with
		$u_{ij}=\pm\widehat{(\delta_j-\delta_i)}
		=\pm(\delta_j-\delta_i)/\norm{\delta_j-\delta_i}$ on zero
		edges and
		$\sgn(e_{ij})$ on the others.
		Thus, for every $\delta>0$ and every null $\ccalS$, both
		$V^{\pm}$ lie in the closure of
		$h(\ccalB_\delta(x)\setminus\ccalS)$, hence in its closed
		convex hull, and so does
		$\tfrac12(V^++V^-)=h(x)$; intersecting over $\delta$ and
		$\ccalS$ gives $h(x)\in\ccalK[h](x)$, proving~(ii).

		For~(iii), fix $s>0$ small and pass to $x_s$ by
		$R_{j_0}\mapsto R_{j_0}\exp(su^\wedge)$; since
		$R_{i_0}=R_{j_0}$ at $x$, the chosen edge carries the exact
		error $e_{i_0j_0}(x_s)=su$. Every other edge incident to $i_0$
		or $j_0$ is nonzero by the cluster hypothesis, remains nonzero
		for sufficiently small $s$,
		and has a signum direction converging to its value at $x$; the
		remaining zero edges are untouched.
		The argument for~(ii) applied at $x_s$, with the chosen edge
		now sign-pinned to $u$, places the corresponding value in
		$\ccalK[h](x_s)$; letting $s\to 0$ and using upper
		semicontinuity of $\ccalK[h]$ with closed values yields
		membership in $\ccalK[h](x)$.
	\end{proof}

	Only part~(i) enters the proof of Theorem~\ref{thm:main};
	parts~(ii)--(iii) certify the simulated selection $\sgn(0)=0$ as a
	Filippov velocity and supply the selection used in
	Remark~\ref{rem:viab}.

	The cluster scale is essential for an every-solution dissipation
	estimate: in the containment of Lemma~\ref{lem:filippov}(i) the
	variables $u_{ij}$ on zero-error edges are unpinned, so agentwise
	substitution leaves, for each zero-error edge, a term
	$-\alpha(s_i-s_j)^\top u_{ij}$ of uncontrolled sign, whereas
	summing over all agents cancels every edge term, including the
	pinned directions on nonzero-error edges. Aggregation over each
	zero-error cluster cancels only the unpinned internal terms and
	retains the pinned intercluster directions $u_{ij}=\sgn(e_{ij})$
	used in Step~2 of the proof of Theorem~\ref{thm:main}.

	\begin{lemma}[Cluster dynamics]\label{lem:clusters}
		Let $x$ be a Filippov solution in $\ccalD_\ccalO$.
		For almost every time $t$, all agents in each cluster
		$C\in\ccalP_0(x(t))$ have a common body angular velocity
		$\bar\omega_C(t)$ satisfying
		\feq{
			|C|\,\bar\omega_C=\alpha S_C-\gamma G_C,
			\ \
			S_C:=\!\sum_{i\in C}\!s_i,\
			G_C:=\!\sum_{i\in C}\!\nabla_{R_i}f_i,
			\label{eq:cluster}
		}
		where $|C|$ denotes the number of agents in $C$.
	\end{lemma}
	\begin{proof}
		First, let $\ccalA\subseteq[0,T]$ be measurable with
		$R_i(t)=R_j(t)$ on $\ccalA$ for a fixed pair $(i,j)$.
		The function $G:=R_i-R_j$ is absolutely continuous and vanishes
		on $\ccalA$; at a.e.\ density point of $\ccalA$ that is also a
		differentiability point, $\dot G=0$, i.e.\
		$R_i\omega_i^\wedge=R_j\omega_j^\wedge$ with $R_i=R_j$
		invertible, so $\omega_i=\omega_j$ a.e.\ on $\ccalA$.
		Second, the zero-edge pattern
		$\{(i,j)\in\ccalE:e_{ij}(t)=0\}$ takes values in the finite set
		$2^{\ccalE}$; for each pattern the corresponding time set is
		Borel, by continuity of $t\mapsto e_{ij}(t)$, and the cluster
		partition is a deterministic function of the pattern
		(Definition~\ref{def:clusters}).
		Applying the first step to every zero edge of every pattern and
		propagating equality along paths within a cluster yields a null
		set outside of which all members of each cluster
		$C\in\ccalP_0(x(t))$ share a velocity $\bar\omega_C(t)$.
		Third, by Lemma~\ref{lem:filippov}(i), a.e.\ velocities take the
		antisymmetric form~\eqref{eq:containment}.
		Summing over $i\in C$: intra-cluster edge contributions cancel
		in antisymmetric pairs; cross-cluster edges have $e_{ij}\neq0$,
		so their $u_{ij}$ are pinned to $\sgn(e_{ij})$ and sum to
		$S_C$ (intra-cluster edges contribute $\sgn(0)=0$ to $S_C$ as
		well).
		Replacing each $\omega_i$, $i\in C$, by $\bar\omega_C$
		gives~\eqref{eq:cluster}.
	\end{proof}

	\subsection{Tree dissipation bound}

	\begin{lemma}[Tree cluster bound]\label{lem:tree_bound}
		Under Assumption~\ref{ass:tree}, whenever $W(x)>0$,
		\feq{
			\sum_{C\in{\ccalP_0(x)}}\|S_C\|\ \geq\ 2 .
			\label{eq:two}
		}
	\end{lemma}
	\begin{proof}
		Each cluster is a connected subtree; contracting every cluster
		of a tree yields again a tree, on $K\geq2$ nodes (since $W>0$),
		whose edges are exactly the nonzero-error edges.
		A tree on $K\geq2$ nodes has at least two leaves.
		If $C$ is a leaf cluster, exactly one nonzero-error edge
		$(i,j)$, $i\in C$, leaves $C$, while all other edges incident
		to members of $C$ are zero-error edges contributing $\sgn(0)=0$;
		hence $S_C=\hat e_{ij}$ and $\|S_C\|=1$.
		Two leaf clusters give~\eqref{eq:two}.
	\end{proof}

	\subsection{Proof of Theorem~\ref{thm:main}}

	\begin{proof}
		\emph{Step 1: Strong invariance and global existence.}
		Let $x(0)\in\ccalC^n$. Since $\ccalC^n\subset\ccalD_\ccalO$,
		a solution exists by the existence statement of
		Section~\ref{sec:protocol}; let $x$ be any solution with
		maximal interval $[0,T_{\max})$.
		Pick $\delta>0$ with $\rho+\delta<\pi/2$ and
		$\overline{\ccalB}_{\rho+\delta}(R_c)\subset\ccalO$, so Lemmas~\ref{lem:filippov} and~\ref{lem:clusters} apply
		verbatim while
		$x(t)\in\overline{\ccalB}_{\rho+\delta}(R_c)^n
		\subset\ccalD_\ccalO$.
		On the compact shell
		$\{\rho\leq d(R,R_c)\leq\rho+\delta\}$ the functions
		$g_i(R):=\hat e_{c}(R)^\top\nabla_{R}f_i(R)$ are $C^1$, hence
		$L$-Lipschitz for some common $L$; writing $\Pi_\rho(R)$ for
		the radial projection to $\partial\ccalC$ along the geodesic
		through $R_c$ and using $g_i(\Pi_\rho(R))\leq0$
		(Assumption~\ref{ass:rad}) with
		$d(R,\Pi_\rho(R))=d(R,R_c)-\rho$,
		\feq{
			g_i(R)\leq L\,\bigl(d(R,R_c)-\rho\bigr)
			\quad\text{on the shell.}
			\label{eq:shellbound}
		}
		Set $r_i(t):=d(R_i(t),R_c)$, $r(t):=\max_i r_i(t)$, and
		$\tau:=\inf\{t\in[0,T_{\max}):r(t)=\rho+\delta\}$, with
		$\inf\varnothing:=T_{\max}$.
		On $[0,\tau)$ the functions $r_i$ and $r$ are absolutely
		continuous. Since each $r_i-r_j$ has derivative zero almost
		everywhere on its zero set, the active derivatives coincide
		almost everywhere, and $\dot r=\dot r_i$ for every active
		index outside a null set.
		At such a $t$ with $r(t)>\rho$, let $C$ be the cluster of an
		active agent, with common attitude $R_C$, $d(R_C,R_c)=r(t)$.
		Every neighbor lies in $\overline{\ccalB}_{r(t)}(R_c)$ and
		every edge leaving $C$ carries a nonzero error, so
		Lemma~\ref{lem:half} at radius $r(t)<\rho+\delta<\pi/2$ gives
		$\hat e_{C,c}^\top\hat e_{ij}\geq0$ and hence
		$\hat e_{C,c}^\top S_C\geq0$; internal selections cancel by
		antisymmetry (Lemma~\ref{lem:clusters}).
		Therefore, by the distance identity~\eqref{eq:deriv}
		and~\eqref{eq:shellbound},
		\nfeq{\dot r = -\hat e_{C,c}^\top\bar\omega_C
			&= -\tfrac{\alpha}{|C|}\hat e_{C,c}^\top S_C
			+\tfrac{\gamma}{|C|}\sum_{i\in C}g_i(R_C)\\
			&\leq\gamma L\,\bigl(r(t)-\rho\bigr).}
		Hence $y:=[r-\rho]_+$ satisfies $\dot y\leq\gamma Ly$ a.e.\ on
		$[0,\tau)$ (on $\{y=0\}$, $\dot y=0$ a.e.), and $y(0)=0$, so
		Gr\"onwall's inequality gives $y\equiv0$ on $[0,\tau)$.
		If $\tau<T_{\max}$, continuity would give both
		$r(\tau)=\rho+\delta$ and $r(\tau)\leq\rho$, a contradiction;
		thus $\tau=T_{\max}$ and $x(t)\in\ccalC^n$ on its maximal
		interval.
		Global extension follows since $\ccalC^n$ is compact with
		positive distance to $\partial\ccalD_\ccalO$ and velocities are
		bounded there. This proves~(i).

		\emph{Step 2: Exact finite-time consensus.}
		Assume~\eqref{eq:gain} and let $x$ be a solution with
		$x(0)\in\ccalC^n$; by Step~1 it is defined on $[0,\infty)$
		with values in $\ccalC^n$.
		Then $w(t):=W(x(t))$ is absolutely continuous ($W$ is locally
		Lipschitz on $\ccalD_\ccalO$ and $x$ is absolutely continuous
		with values in the compact $\ccalC^n$).
		Fix a.e.\ $t$.
		On edges with $e_{ij}(t)\neq0$,
		the distance identity~\eqref{eq:deriv} gives
		$\tfrac{d}{dt}\|e_{ij}\|=\hat e_{ij}^\top(\omega_j-\omega_i)$;
		on the time set where an edge error vanishes, the absolutely
		continuous function $\|e_{ij}(\cdot)\|$ is zero, so its
		derivative vanishes a.e.\ there.
		Summing over edges and regrouping endpoint contributions,
		\nfeq{
			\dot w
			= -\sum_{i\in\ccalV}s_i^\top\omega_i
			= -\!\!\sum_{C\in{\ccalP_0(x(t))}}\!\! S_C^\top\bar\omega_C
			\quad\text{a.e.},
		}
		using Lemma~\ref{lem:clusters} and $\sum_{i\in C}s_i=S_C$.
		Substituting~\eqref{eq:cluster},
		$\|G_C\|\leq|C|M$, and the Cauchy--Schwarz bound
		$\sum_C\|S_C\|^2/|C|\geq\Sigma^2/n$ with
		$\Sigma:=\sum_C\|S_C\|$,
		\nfeq{
			\dot w
			\leq -\frac{\alpha}{n}\Sigma^2+\gamma M\Sigma
			=: \psi(\Sigma)
			\quad\text{a.e.\ on }\{w>0\}.
		}
		The concave parabola $\psi$ has vertex at
		$\Sigma=n\gamma M/(2\alpha)<1<2$
		by~\eqref{eq:gain}, so $\psi$ is decreasing on $[2,\infty)$;
		Lemma~\ref{lem:tree_bound} gives $\Sigma\geq2$ on $\{w>0\}$,
		hence
		$\dot w\leq\psi(2)=-\tfrac{2}{n}(2\alpha-\gamma nM)=:-c<0$
		a.e.\ on $\{w>0\}$.
		Define $T^*:=\inf\{t\geq0:w(t)=0\}$. If $w(0)>0$, then
		$w>0$ on $[0,T^*)$, and integration gives
		$0\leq w(t)\leq w(0)-ct$ for every $t<T^*$; hence
		$T^*\leq w(0)/c$. If $w(0)=0$, then $T^*=0$.
		Moreover, $w$ cannot leave zero afterwards: any excursion
		interval $(t_2,t_1]$ with $w(t_2)=0$ and $w>0$ on $(t_2,t_1]$
		would give $w(t_1)\leq-c(t_1-t_2)<0$.
		Finally, $w=0$ forces every edge error to vanish, and
		connectivity gives consensus for all $t\geq T^*$; substituting
		$c$ yields~\eqref{eq:settle}. This proves~(ii).

		\emph{Step 3: Post-consensus optimization.}
		By Step~2, agreement persists: $R_i(t)=R_j(t)$ for all $i,j$
		and all $t\geq T^*$, so $\bar R(t):=R_i(t)$ is well defined,
		and for $t\geq T^*$ the trajectory evolves in the consensus set
		$\ccalM=\{(R,\ldots,R):R\in\ccalC\}$, a submanifold with
		boundary, so its a.e.\ body velocities are of the common form
		$(\bar\omega,\ldots,\bar\omega)$.
		By Lemma~\ref{lem:filippov}(i), the same velocities also admit
		the antisymmetric form~\eqref{eq:containment}; summing over all
		agents, every edge term cancels in antisymmetric pairs and
		$n\bar\omega=-\gamma\nabla_{\bar R}F$.
		This proves~\eqref{eq:sliding}; its right-hand side is
		locally Lipschitz, so~\eqref{eq:sliding} uniquely determines the
		common trajectory from $\bar R(T^*)$. The sliding velocity is
		radially admissible at $\partial\ccalC$, since
		$\hat e_c^\top(-\tfrac{\gamma}{n}\nabla F)
		=-\tfrac{\gamma}{n}\sum_i\hat e_c^\top\nabla f_i\geq0$ by
		Assumption~\ref{ass:rad}.
		For the rate, set $V_2:=\tfrac12 d^2(\bar R,R^*)$; by the
		distance identity~\eqref{eq:deriv} with the static endpoint
		$R^*$ and $e_*:=(\log(\bar R^\top R^*))^\vee$,
		$\dot V_2=-e_*^\top\bar\omega
		=\tfrac{\gamma}{n}\inner{\nabla_{\bar R}F}{e_*}
		\leq-\tfrac{\gamma}{n}\mu_F d^2=-\tfrac{2\gamma\mu_F}{n}V_2$
		a.e., using the interior-minimizer
		inequality~\eqref{eq:sc} (Assumption~\ref{ass:conv}(a)--(b)).
		Gr\"onwall's inequality and $d=\sqrt{2V_2}$
		give~\eqref{eq:rate}. This proves~(iii).
	\end{proof}

	\begin{remark}[Invariance despite pointwise tangency
		failure]\label{rem:viab}
		In body coordinates, the Bouligand tangent cone of the domain
		splits as
		$T_{\ccalC^n}(x)=\prod_iT_\ccalC(R_i)$ with
		\nfeq{
			T_\ccalC(R_i)=
			\begin{cases}
				\mbR^3, & d(R_i,R_c)<\rho,\\
				\{\omega:\hat e_{i,c}^\top\omega\geq0\},
				& d(R_i,R_c)=\rho.
			\end{cases}
		}
		For the $n=2$ tree consisting of the single edge $\{1,2\}$
		with $R_1=R_2\in\partial\ccalC$, Lemma~\ref{lem:filippov}(iii)
		places the free selection $u_{12}=-\hat e_{1,c}$ on the zero
		edge inside $\ccalK[h](x)$, and no other edge is incident, so
		\nfeq{
			\hat e_{1,c}^\top\omega_1
			=-\alpha-\gamma\,\hat e_{1,c}^\top\nabla_{R_1}f_1
			\leq-\alpha+\gamma M<0
		}
		by the gain condition~\eqref{eq:gain}, which for $n=2$ reads
		$\alpha/\gamma>M$: the selected velocity has a strictly
		outward radial component, hence
		$\ccalK[h](x)\not\subseteq T_{\ccalC^n}(x)$ and the standard
		pointwise sufficient condition for strong invariance is
		inapplicable.
		Membership in $\ccalK[h](x)$ does not imply that the selected
		vector can be sustained along a solution: on the coincidence
		set the common-velocity identity of Lemma~\ref{lem:clusters}
		holds almost everywhere, and the maximum-radius estimate of
		Step~1 excludes every outward excursion, as for the scalar
		analogue $\dot z\in-\Sgn(z)$, $z\in\mbR$, whose value at $z=0$
		contains nonzero elements while its only solution from $z=0$
		is $z\equiv0$. The operating ball is thus strongly invariant
		although the pointwise tangency test fails; compare the weaker
		existential notion of viability~\cite{aubin2009viability}.
		Note that Step~1 uses neither acyclicity nor strong convexity.
	\end{remark}

	\begin{remark}[Role of the tree topology]\label{rem:tree_vs_general}
		The proof of Lemma~\ref{lem:tree_bound} is specific to trees:
		contracting zero-error clusters again produces a tree, whose leaf
		clusters each carry exactly one intercluster edge and contribute
		one unit vector to $\sum_C\|S_C\|$, independently of the
		attitudes, the operating radius, and the tree shape. With cycles,
		several active edges may leave one cluster and partially cancel,
		so a different estimate would be needed. On a general connected
		graph the protocol can be run over a spanning tree computed once
		in a distributed manner, at the cost of discarding the remaining
		edges. This independence concerns~\eqref{eq:two} only; the
		settling bound~\eqref{eq:settle} still depends on $M$ and $W(0)$.
	\end{remark}

	\section{Numerical Validation}\label{sec:sim}

	\subsection{Setup and numerical protocol}\label{sec:sim_setup}

	Numerical trajectories use the geometric-Euler update
	$R_i\leftarrow R_i\exp(h_{\mathrm{num}}\omega_i^\wedge)$ with
	$h_{\mathrm{num}}=10^{-4}$~s and the nominal selection $\sgn(0)=0$,
	an admissible Filippov velocity by Lemma~\ref{lem:filippov}(ii).
	We set $R_c=I_3$, $\rho=0.6$, $r_0=0.35$, and
	$k=(1.2,0.9,1.0,1.1,0.8)$; targets and initial attitudes are obtained
	by exponentiating samples drawn uniformly from the tangent balls of
	radii $r_0$ and $0.92\rho$, respectively with random seed~$7$.
	The three nonisomorphic five-node trees are the star $K_{1,4}$
	(hub agent~$1$), the path $P_5$, and the T-tree with edges
	$\{1,2\},\{2,3\},\{3,4\},\{3,5\}$.
	Here $M=\max_ik_i(\rho+r_0)=\accM$, so~\eqref{eq:gain} requires
	$\alpha/\gamma>nM/2=\accThrTwo$; with $\gamma=0.5$ and $\alpha=2.0$
	($\accSafety$ \(\times\) the threshold),
	$c=\tfrac2n(2\alpha-\gamma nM)=\accC$~rad/s and, from
	Remark~\ref{rem:class}, $\mu_F\approx\accMuF$ and
	$\gamma\mu_F/n\approx\accRate$~s$^{-1}$.
	Let $T_{\mathrm{bd}}:=nW(0)/[2(2\alpha-\gamma nM)]$ denote the bound
	in~\eqref{eq:settle}.
	A fixed-step trajectory chatters near the sliding set, so we report
	a tolerance crossing $T_{\mathrm{tol}}$ rather than identify it with
	$T^*$: for a run with horizon $T_{\mathrm{end}}$ ($6$~s for the base
	runs, $1.2$~s for the random trees), $T_{\mathrm{tol}}$ is the first
	sampled time after which
	$W\leq\mathrm{tol}:=12\alpha h_{\mathrm{num}}|\ccalE|$ ($=\accTol$
	here) at all remaining samples. Per edge, $12\alpha
	h_{\mathrm{num}}$ is about twice the largest one-step change of an
	edge error, $\bigl(\alpha(|\ccalN_i|+|\ccalN_j|)+2\gamma M\bigr)
	h_{\mathrm{num}}\leq5.6\,\alpha h_{\mathrm{num}}$ on the five-node
	trees, so the tolerance sits just above the chattering band.
	Put $\varepsilon:=\mathrm{tol}/|\ccalE|=\accEps$, let
	$\ccalP_\varepsilon$ be the components of the edges with
	$\|e_{ij}\|<\varepsilon$, and for $q_C:=\min C$ define
	$S_{C,\varepsilon}:=\sum_{i\in C}\sum_{j\in\ccalN_i\setminus C}
	\Ad_{R_{q_C}^{\top}R_i}\sgn(e_{ij})$ and
	$\Sigma_\varepsilon:=\sum_{C\in\ccalP_\varepsilon}\|S_{C,\varepsilon}\|$
	(another representative rotates each sum, leaving the norms
	unchanged). This is a discrete proxy, not the exact cluster
	partition of a Filippov trajectory.
	Halving $h_{\mathrm{num}}$ and rescaling $\mathrm{tol}$ changes
	$T_{\mathrm{tol}}$ by at most $3$~ms and
	$D_6:=\max_id(R_i(6),R^*)$ by at most $\accHalfDiffMax\%$.

	\subsection{Phase 1: finite-time consensus}\label{sec:sim_p1}

	Fig.~\ref{fig:phase1}(a) compares $W$ on the star with the envelope
	$W(0)-ct$ of Theorem~\ref{thm:main}(ii) and with proportional consensus
	feedback $\omega_i=\alpha\sum_je_{ij}-\gamma\nabla f_i$ at the same
	gains. Here $T_{\mathrm{tol}}=\accTtolStar$~s, whereas
	$T_{\mathrm{bd}}=\accBoundStar$~s; the sharp entry into the tolerance
	band is numerically consistent with finite-time convergence, while
	the proportional law retains $W(6)=\accPropWsix$.
	Fig.~\ref{fig:phase1}(b) shows $\Sigma_\varepsilon$: whenever
	$W>\mathrm{tol}=|\ccalE|\varepsilon$ the $\varepsilon$-quotient is
	nontrivial and the leaf argument gives $\Sigma_\varepsilon\geq2$, so
	the informative feature is attainment of the bound. On the star,
	$\Sigma_\varepsilon$ stays within $0.02$ of~$2$ from
	$t=\accSigTwoTimeStar$~s until $T_{\mathrm{tol}}$, with exactly two
	$\varepsilon$-clusters recorded; the T-tree does so from
	$t=\accSigTwoTimeTtree$~s, and the path stays within
	$\accSigNearExcessPath$ of~$2$ from $t=\accSigNearTwoTimePath$~s
	until $T_{\mathrm{tol}}$. A least-squares fit of $W$ over that
	interval on the star gives $-\dot w\approx\accWslopeStratumStar$~rad/s
	versus $c=\accC$~rad/s, consistent with conservatism in the
	subsequent Cauchy--Schwarz and uniform-gradient estimates.

	\begin{figure}[!t]
		\centering
		\includegraphics[width=\columnwidth]{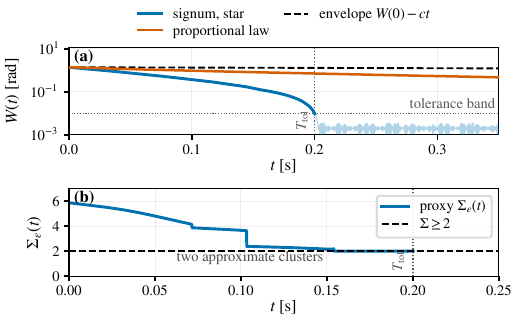}
		\caption{Phase~1 on the star at $\rho=0.6$ (Theorem~\ref{thm:main}(ii),
				Lemma~\ref{lem:tree_bound}). (a)~$W(t)$, log scale, with the
				envelope $W(0)-ct$, the tolerance band, and the proportional
				law (faint tail: chattering after $T_{\mathrm{tol}}$).
				(b)~Tolerance-based cluster-sum proxy
				$\Sigma_\varepsilon(t)$ against the constant~$2$.}
		\label{fig:phase1}
	\end{figure}

	\emph{Gain and radius sweeps.}
	Fig.~\ref{fig:sweeps}(a) varies $\alpha$ from $\accSweepAlphaMin$ to
	$\accSweepAlphaMax$ (star shown; the other trees differ by at most
	$\accSweepTreeSpread\%$). Every crossing above the sufficient
	threshold $\gamma nM/2=\accThr$ lies below $T_{\mathrm{bd}}(\alpha)$,
	which diverges at the threshold while $T_{\mathrm{tol}}$ decreases
	only from $\accSweepTtolLoStar$ to $\accSweepTtolHiStar$~s. Crossings
	also occur below that threshold, so the sufficient
	condition~\eqref{eq:gain} is conservative for the displayed
	instances.
	Fig.~\ref{fig:sweeps}(b) varies $\rho\in\{0.45,0.6,0.75,0.9,1.2,1.5\}$
	with the targets fixed and $\alpha/\gamma=1.4\,nM/2$ rescaled with
	$M$. On all three trees
	$T_{\mathrm{tol}}/T_{\mathrm{bd}}\leq\accRadRatioMaxAll$, and every
	recorded iterate remains in the prescribed ball.

	\emph{Random-tree trials.}
	At $\rho=0.6$ and $\gamma=0.5$ we also test ten labeled trees sampled
	by Pr\"ufer sequences with random seed~$2027$ for each $n\in\{5,8,12\}$,
	resampling targets and initial attitudes independently (tree $q$ of
	size $n$ uses seed $1000+10n+q$) and, for $n>5$, weights uniformly on
	$[0.8,1.2]$ rounded to $0.01$. With $\alpha/\gamma=1.4\,nM/2$ and a
	$1.2$~s horizon, $T_{\mathrm{tol}}/T_{\mathrm{bd}}$ ranges over
	$[\accRndRatioLoFive,\accRndRatioHiFive]$,
	$[\accRndRatioLoEight,\accRndRatioHiEight]$, and
	$[\accRndRatioLoTwelve,\accRndRatioHiTwelve]$, respectively. Because
	costs and initial conditions also vary, these ranges are not
	isolated tree-shape effects.

	\begin{figure}[!t]
		\centering
		\includegraphics[width=\columnwidth]{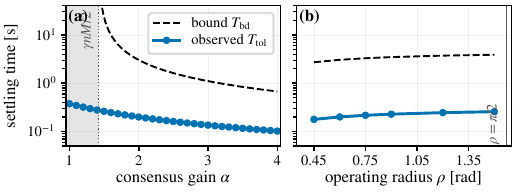}
		\caption{Settling time on the star versus (a)~gain $\alpha$ at
				$\rho=0.6$ and (b)~radius $\rho$ at $\alpha/\gamma=1.4\,nM/2$:
				observed $T_{\mathrm{tol}}$ (blue) and bound
				$T_{\mathrm{bd}}$ of~\eqref{eq:settle} (dashed). Shaded
				in~(a): the gain condition~\eqref{eq:gain} fails.}
		\label{fig:sweeps}
	\end{figure}

	\subsection{Phase 2: sliding flow and exponential rate}\label{sec:sim_p2}

	Fig.~\ref{fig:phase2}(a) compares $\max_id(R_i,R^*)$ on the star,
	with $R^*$ the weighted Karcher mean computed offline, against the
	guaranteed distance rate $\gamma\mu_F/n\approx\accRate$~s$^{-1}$
	of~\eqref{eq:rate}. The fitted decay rate on $[1,6]$~s is
	$\accRateFitStar$~s$^{-1}$ (Table~\ref{tab:summary}), consistent
	with conservatism of the global modulus $\mu_F$.
	For Fig.~\ref{fig:phase2}(b), each velocity is expressed in agent~$1$'s
	body frame as $\omega_i^{(1)}=R_1^\top R_i\omega_i$ and averaged over
	the nonoverlapping windows
	$[T_{\mathrm{tol}}+k\Delta,T_{\mathrm{tol}}+(k+1)\Delta)$, $k\geq1$,
	contained in $[T_{\mathrm{tol}},2]$~s (the first window after the
	crossing is discarded); the worst per-agent residual over these
	windows to the similarly averaged target
	$-\tfrac{\gamma}{n}\nabla F(R_1)$ decreases approximately as
	$\Delta^{-1}$, from $\accSlideNarrowOneAll$~rad/s at $\Delta=5$~ms to
	$\accSlideWideOneAll$~rad/s at $\Delta=0.5$~s, supporting the sliding
	identification~\eqref{eq:sliding} in a window-averaged numerical
	sense.

	\begin{figure}[!t]
		\centering
		\includegraphics[width=\columnwidth]{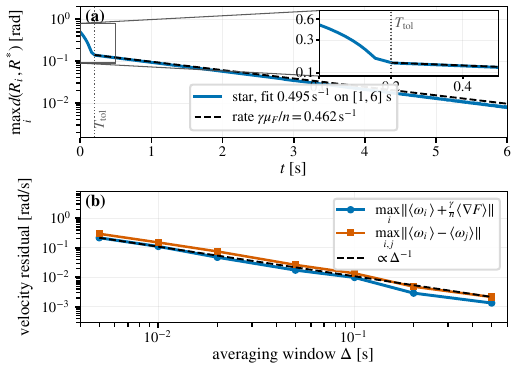}
		\caption{Phase~2 on the star (Theorem~\ref{thm:main}(iii)).
				(a)~Distance to the minimizer with $T_{\mathrm{tol}}$ and a
				reference line of slope $-\gamma\mu_F/n$ anchored at
				$T_{\mathrm{tol}}$; inset: $t\in[0,0.5]$~s.
				(b)~Worst-window per-agent residual to the sliding
				prediction~\eqref{eq:sliding} and pairwise mismatch versus
				window $\Delta$, against a $\Delta^{-1}$ decay.}
		\label{fig:phase2}
	\end{figure}

	\subsection{Invariance at the boundary}\label{sec:sim_inv}

	The base runs have minimum sampled boundary margin
	$\rho-\max_id(R_i,R_c)=\accMinMargStar$~rad (at $t=0$), so
	Fig.~\ref{fig:invariance} uses boundary initializations on the star
	with $h_{\mathrm{num}}=\accHb$~s. In panel~(a) all agents start on
	$\partial\ccalC$; their initial radial derivatives
	$\dot r_i(0)=-\hat e_{i,c}^\top\omega_i(0)=(\accRdotsB)$~rad/s are
	strictly negative, most negative at the hub in this run, and the
	minimum margin over the sampled times $t\geq h_{\mathrm{num}}$ is
	$\accMinMargB$~rad.
	In panel~(b) agents~$1$ (hub) and~$2$ start coincident on the
	boundary, a coincidence analogous to the configuration of
	Remark~\ref{rem:viab}.
	Under the nominal zero-edge selection agent~$2$ has initial radial
	derivative $\accLeafGrad$~rad/s, inward as Assumption~\ref{ass:rad}
	requires. In the Euler discretization, the initial velocity
	mismatch ($\accSepZero$~rad/s) opens the edge at the first step,
	after which the signum feedback restores approximate agreement.
	The minimum margin over $t\geq h_{\mathrm{num}}$ is
	$\accMinMargC$~rad; the tolerance crossings are $\accTtolB$
	and $\accTtolC$~s.

	\begin{figure}[!t]
		\centering
		\includegraphics[width=\columnwidth]{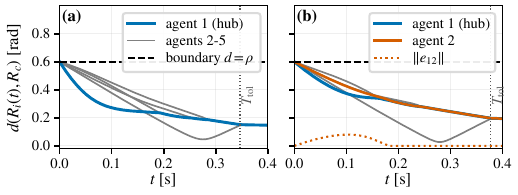}
		\caption{Invariance (Theorem~\ref{thm:main}(i)), star: per-agent radii
				$d(R_i(t),R_c)$ with the boundary $d=\rho$ (dashed). (a)~All
				agents on $\partial\ccalC$. (b)~Agents~$1$ (hub) and~$2$
				coincident on $\partial\ccalC$; dotted: $\|e_{12}(t)\|$; gray: agents $3$--$5$.}
		\label{fig:invariance}
	\end{figure}

	\begin{table}[!t]
		\caption{Theoretical references (ideal continuous-time solution)
				versus observed values (stated discretization and tolerances),
				base runs at $\rho=0.6$; times in s, rates in s$^{-1}$,
				residuals in rad/s.}
		\label{tab:summary}
		\centering
		\scriptsize
		\setlength{\tabcolsep}{2.5pt}\renewcommand{\arraystretch}{0.92}
		\resizebox{\columnwidth}{!}{%
			\begin{tabular}{@{}lcccc@{}}
				\hline
				Quantity & Theoretical reference & Star & Path & T-tree \\
				\hline
				Settling time ($T_{\mathrm{tol}}$ observed) & $T^*\leq T_{\mathrm{bd}}$ & $\accTtolStar$ & $\accTtolPath$ & $\accTtolTtree$ \\
				$\min_{W(t)>\mathrm{tol}}\Sigma_\varepsilon(t)$ & $\geq2$ & $\accSigminStar$ & $\accSigminPath$ & $\accSigminTtree$ \\
				Phase-2 rate, fit & $\gamma\mu_F/n\approx\accRate$ & $\accRateFitStar$ & $\accRateFitPath$ & $\accRateFitTtree$ \\
				sliding residual, $\Delta{=}0.5$~s & $0$ & $\accSlideWideOneStar$ & $\accSlideWideOnePath$ & $\accSlideWideOneTtree$ \\
				pairwise mismatch, $\Delta{=}0.5$~s & $0$ & $\accSlideWideTwoStar$ & $\accSlideWideTwoPath$ & $\accSlideWideTwoTtree$ \\
				\hline
		\end{tabular}}
	\end{table}

	\section{Conclusion}\label{sec:conclusion}

	We established a two-stage convergence result for an intrinsic
	signum--gradient protocol on $\SO(3)$ over undirected trees.
	Cluster aggregation yields exact finite-time consensus for every
	Filippov solution, after which the common attitude follows an
	aggregate Riemannian gradient flow and converges exponentially
	to the interior minimizer. A separate boundary condition ensures
	strong invariance of the operating ball and is satisfied by the
	weighted squared-distance costs considered here. The guarantees
	concern the ideal continuous-time protocol; general graphs and
	sampled-data implementations remain open directions.

	\balance
	\bibliographystyle{IEEEtran}
	\bibliography{references}

\end{document}